\documentclass[12pt]{amsart}
\usepackage{amsmath,amsfonts,amsthm,amsopn,cite,mathrsfs}
\usepackage{epsfig,verbatim}
\usepackage{subfigure}
\newcommand{\tfa}{time-frequency analysis}

\newcommand{\fif}{if and only if}

\newtheorem{tm}{Theorem}%[section]
\newtheorem{lemma}[tm]{Lemma}
\newtheorem{prop}[tm]{Proposition}

\newcommand{\rem}{\noindent\textsl{REMARK:}}

  \theoremstyle{definition}
  \newtheorem{definition}{Definition}

\newcommand{\beqa}{\begin{eqnarray*}}
\newcommand{\eeqa}{\end{eqnarray*}}

\newcommand{\field}[1]{\mathbb{#1}}
\newcommand{\bN}{\field{N}}        %  natural numbers
\newcommand{\bZ}{\field{Z}}        %  whole numbers
\newcommand{\bT}{\field{T}}        %
  \def\cB{\mathcal{B}}

  \def\cU{\mathcal{U}}
  \def\cA{\mathcal{A}}

  \def\cC{\mathcal{C}}

  \def\cO{\mathcal{O}}

  \def\cZ{\mathcal{Z}}

\def\<{\left<}
\def\>{\right>}

\def\mv1{M_v^1}

\newcommand{\T}{\mathbb{T}}
\newcommand{\Z}{\mathbb{Z}}
\newcommand{\zn}{\Z ^n}
\newcommand{\eltw}{\ell ^1 (\zn , \theta )}
\newcommand{\twist}{\, \natural \, }
\newcommand{\ic}{inverse-closed}

\begin{document}
\begin{abstract}
We give a systematic construction of inverse-closed (Banach)
subalgebras in general higher-dimensional non-commutative tori. % The
% construction is motivated by the investigation of symmetric group
% algebras.
\end{abstract}

\title[Inverse-Closed  Subalgebras of
   Noncommutative Tori]{Inverse-Closed Banach Subalgebras of Higher-Dimensional
   Noncommutative Tori}
\author{Karlheinz Gr\"ochenig}
\address{Faculty of Mathematics \\
University of Vienna \\
Nordbergstrasse 15 \\
A-1090 Vienna, Austria}
\email{karlheinz.groechenig@univie.ac.at}
\author{Michael Leinert}
\address{Institut f\"ur Angewandte Mathematik, Fakult\"at f\"ur Mathematik,
Im Neuenheimer Feld 288, D-69120 Heidelberg, Germany}
\email{leinert@math.uni-heidelberg.de}

%\date{\today}
\thanks{K.\ G.\ was
   supported in part by the  project P22746-N13  of the
Austrian Science Foundation (FWF) and  by the Marie-Curie Excellence Grant
   MEXT-CT 2004-517154}
\maketitle

\section{Introduction}

Let $\cA \subseteq \cB $ be two algebras with common identity. Then
$\cA  $ is called inverse-closed  in $\cB $, if $a\in \cA $ and $a^{-1}
\in \cB $ implies that $a^{-1} \in \cA $. This property is a
generalization of Wiener's Lemma for absolutely convergent Fourier
series and occurs abundantly  in many branches of  mathematical
analysis. The range of applications covers  numerical analysis,
pseudodifferential operators,  frame theory, and %  Inverse-closed 
%Banach algebras occur in numerical analysis
% in the study of the off-diagonal decay of matrices, in the theory of
% pseudodifferential operators and their spectral invariance when they
% act on different function spaces; in phase-space analysis for the
% construction of discrete coherent state expansions;
  last not least  non-commutative tori. See~\cite{Gr10} for a survey
  of many versions of Wiener's Lemma and applications of
  inverse-closedness.

The main result concerning inverse-closed subalgebras of
non-commutative tori is the density theorem. It  states that the $K$-groups of
a  non-commutative torus and of all its  inverse-closed subalgebras are
isomorphic. Similarly the  stable rank of an inverse-closed
subalgebra coincides with the   stable rank of the large
algebra~\cite{Bad98}.  Usually  the existence of an inverse-closed subalgebra
is taken for granted and is  the starting point for the
theory. Also, mostly   Fr\'echet subalgebras  are considered rather than
Banach subalgebras (because Fr\'echet algebras model ``smooth'' 
noncommutative tori).

In this paper we investigate a systematic construction of Banach
subalgebras of  non-commutative tori in higher dimensions. We will 
characterize
all inverse-closed Banach subalgebras of the form $\ell ^1_v$, where
$v$ is a weight function on $\bZ ^n$.   By choosing
weights of subexponential growth, we even  construct a Banach
subalgebra  that is contained in the ordinary smooth noncommutative
torus. For certain noncommutative tori with an even number of
generators these results were already obtained in~\cite{GL04}.
This work is motivated by a question of   N.~C.~Phillips who %  at the  workshop on
% ``Operator Methods, Fractals, and Wavelets'' in Banff 2006. He
asked us  whether the results in\cite{GL04} also  hold for arbitrary
non-commutative tori 
in higher dimensions.

Our methods are drawn from abstract harmonic analysis, in particular
the investigation of projective representations and twisted
convolution algebras in the school of Leptin and Ludwig.  

Let us mention that in some areas an  inverse-closed subalgebra is 
also called a spectral subalgebra,  a
local subalgebra, or a
full algebra.  If
$\cA $ is inverse-closed in $\cB $,  then $\cA $ is called spectrally
invariant in $\cB $ or (under standard conditions)  invariant under holomorphic
calculus; $(\cA , \cB )$ is called a Wiener pair.%  Sometimes
% the word spectral permanence is used.

%%%%%%%%%%%%%
%%%%%%%%%%
%%%%%%%%%%

\section{Higher-Dimensional Non-Commutative Tori}

We first give a description of non-commutative tori in higher
dimensions and explain the link to harmonic analysis.

Let $\T$ denote the unit circle. Let $U_{1}, \ldots , U_{n}$ be
unitary symbols  satisfying the commutation relations
$$
U_{j} U_{k} = \theta_{jk} U_{k} U_{j} \, ,
$$ where
$\theta_{jk} \in \T$.
Since $U_j U_k = \theta _{jk} U_kU_j = \theta _{jk} \theta _{kj}
U_jU_k$,  we have $\theta _{kj} = \overline{\theta _{jk}}$ and thus
the matrix $\theta  = (\theta _{jk})_{j,k= 1, \dots ,n}$ is hermitean.

The non-commutative torus $C^*(\theta )$ is the universal
$C^*$-algebra generated by the unitaries $U_j, j=1, \dots ,n$. To
obtain a concrete and workable representation, we interpret
$C^*(\theta )$ as a twisted group $C^*$-algebra of $\Z ^n$.

% Letting $U=U_{1} \cdots U_{n}$ and
Using multi-index notation with  $U^{l} =
U_{1}^{l_{1}} \cdots U_{n}^{l_{n}}$ for $l\in \Z^{n}$, we obtain
\begin{equation}\label{eq:1}
U^{l}U^{m} = \sigma (l,m)U^{l+m}�\, \mbox{ for } l,m\in \Z^{n},
\end{equation}
where $\sigma (l,m)\in \T$. In fact,  repeated application of the
commutation rules yields the expression
\begin{equation}
   \label{eq:1a}
\sigma (l,m) = \left(\prod_{j=1}^{n-1} 
\theta_{n,j}^{m_{j}}\right)^{l_{n}}\left(\prod_{j=1}^{n-2}\theta_{n-1,j}^{m_{j}})\right)^{l_{n-1}}   
\cdots \left(\prod_{j=1}^{1} \theta_{2,j}^{m_{j}}\right)^{l_{2}} =
\prod _{1\leq j<k \leq n} \theta _{k,j}^{l_k m_j}\, .
\end{equation}
Since $U^{0} = U_1^0 \dots U_n ^0 = \mathrm{I}$, \eqref{eq:1} implies
$\sigma (0,m) = \sigma (m,0) = 1$ which is consistent with  \eqref{eq:1a}. We also have $\sigma (-l,m) = \sigma (l,-m)
= \overline{\sigma (l,m)}$. 

Since we require the multiplication to be associative, we have
\begin{equation}\label{eq:2}
\sigma(l,m) \sigma(l+m,p) = \sigma(l,m+p)\sigma(m,p) \, \mbox{ for } 
l,m,p\in\Z^{n}.
\end{equation}
  For $f$ and  $g\in \ell ^{1}(\Z^{n})$  we define the twisted
  convolution $f \natural_{\theta}  g$ or simply $f\natural g$ by
\begin{equation}
   \label{eq:c3}
f \, \natural _\theta \, g(x) =
\sum_{y\in\Z^{n} } f(y) g(x-y) \, \sigma (y,x-y) \qquad x \in
\Z ^n \, .
\end{equation} 
The involution $f \mapsto f^{\ast}$  is defined  by $f^{\ast}(x) =
\overline{\sigma(x,-x)}\, \overline{f(-x)}$ for $x\in \Z^{n}$.
For the special case of ``Dirac'' functions  $\delta _y = \chi_{\{y\}}$ we have 
\begin{equation} \label{eq:cx}
\delta _y \twist \delta _z = \sigma (y,z) \delta _{y+z} \quad \text{
  and } \quad \delta _y^* = \overline{\sigma (-y,y)} \delta _{-y}
\qquad y,z \in \mathbb{Z}^n \, .
\end{equation}
We also note that $\delta _y \twist \delta _y^* =  \sigma (y,-y)
\overline{\sigma   (-y,y)} \delta _0 = \delta _0$ and 
$\delta _y ^* \twist \delta _y  = \overline{\sigma
  (-y,y)} \sigma (-y,y) \delta _0 = \delta _0$, so $\delta _y$
is unitary for every $y\in \mathbb{Z}^n$. 

 Then  $(\ell
^{1}(\Z^{n}), \natural_\theta , ^\ast )$ is a Banach
$\ast$-algebra, which we denote by $\ell ^1(\Z ^n , \theta )$. This
fact can be checked directly,   but  also follows from the reasoning below.

Following~\cite{sch26} and \cite{mackey58}, we  define a central extension
$G$ of $\T$ by $\Z^{n}$ as follows. Let  $G=\{(x, \xi) : x\in \Z^{n}, \xi
\in \T\}$ with multiplication $(x,\xi)(y,\eta) = (x+y,
 \sigma(x,y)\xi\eta )$. Then $G$ is a nilpotent  group with
neutral element $e 
=(0,1)$ and inverse $(x,\xi)^{-1}=(-x, \overline{\sigma(x,-x) \xi }) $.
The Haar measure on $G$ is $\int _G f(a) \, da = \sum _{x\in \Z ^n}
\int _\bT f(x,\xi ) \, d\xi $, and the group convolution $ \star $ on 
$G$ is defined
with respect to this  measure.

  For $f\in \ell ^{1}(\Z^{n})$ we define $f^{\circ}\in
L^{1}(G)$ by $f^{\circ}(x,\xi ) = f(x)\overline{\xi }$. This extension satisfies
the following properties.

\begin{lemma} \label{l1}
   The mapping $\circ : \ell ^1(\Z ^n )  \to  L^1(G)$ is an isometric
   $*$-homomorphism from $\ell ^1(\bZ ^n, \theta )$ into $L^1(G)$.
\end{lemma}

\begin{proof}
  We have
$$\|f^{\circ}\|_{1}
= \int_{G} |f^{\circ}(a) | da = \sum_{x\in\Z^{n}} \int_{\T}
|f(x)\overline{\xi }
|\, d\xi = \sum_{x\in\Z^{n}} |f(x)| = \|f\|_1 \, ,
$$ so $f \mapsto f^{\circ}$ is an  isometry. This map is compatible
with the involution, since
\begin{eqnarray*}
(f^{\ast})^{\circ}(x,\xi)&=& f^{\ast}(x)\overline{\xi } 
= \overline{\sigma(x,-x) f(-x) \xi } \\
&=& \overline{f^{\circ}(-x, \overline{\sigma  (x,-x) \xi } )}= \overline{f^{\circ}((x,\xi )^{-1})}= (f^{\circ})^{\star} (x,\xi)\, .
\end{eqnarray*}
For the homomorphism property we first write
\begin{eqnarray*}
  (f \natural g)^{\circ}(x,\xi ) &=& (f\natural g)(x)\,  \overline{\xi
    }=\\
&=& \sum_{y\in \bZ ^n} f(y) g(x-y) \sigma(y,x - y) \overline{\xi } \, .
%\cdot \int_{\T} \eta \overline{\eta} d\eta \, . % =\\
% &=& \sum_{y\in \bZ ^n} \int_{\T} f^{\circ} (y, \eta) g^{\circ}(x - y, 
% \overline{\sigma(y,x-y)} \xi \overline{\eta}) \, d\eta.
\end{eqnarray*}
On the other hand, using that 
$$
(y,\eta )^{-1} (x,\xi ) = (-y, \overline{\sigma (y,-y) \eta}) (x,\xi )
= (x-y, \xi \overline{\eta} \overline{\sigma (y,-y) }\sigma (-y,x)) \,
,
$$
we obtain 
\begin{eqnarray*}
   (f^{\circ} \star g^{\circ} )(x, \xi ) & = &  \sum_{y}\int_{\T} f^{\circ}(y,\eta 
) g^{\circ}\big((y,\eta )^{-1} (x,\xi )\big)\, d\eta\\ 
&=& \sum_{y}\int_{\T} f(y) \overline{\eta} g(x-y) \overline{\xi} \eta
\sigma (y,-y)\overline{ \sigma (-y,x)} \, d\eta \, .
\end{eqnarray*}
Using (\ref{eq:2}) with $(l,m,p) = (y,-y,x)$ and $\sigma(0,x)=1$, we
have
$$
\sigma (y,-y) \sigma (0,x) = \sigma (y,-y+x) \sigma (-y,x)
$$
or $\sigma (y,-y)\overline{ \sigma (-y,x)} = \sigma (y,x-y)$. Comparing the formulas, we see that
$$
(f \, \natural \, g)^{\circ} = f^\circ \star g^{\circ} \, .
$$
% \begin{eqnarray*}
% (x - y, \overline{\sigma(y,x-y) }\xi \overline{\eta})& = &(-y, 
% \overline{\sigma(y, -y)\eta})(x, \xi 
% \sigma(y,-y)\overline{\sigma(-y,x)}\overline{\sigma(y,x-y)})\\
% &=&(-y, \overline{\sigma(y, -y)\eta})(x, \xi) \, .
% \end{eqnarray*}
% So we arrive at
% \begin{eqnarray*}
% (f\natural g)^{\circ}(x,\xi )& =& \sum_{y}\int_{\T} f^{\circ}(y,\eta 
% ) g^{\circ}((y,\eta )^{-1} (x,\xi ))\, d\eta\\
% & =& \int_{G} f^{\circ} (a) g^{\circ}(a^{-1}(x,\xi ))\, da \, \\ % 
% \mbox{ (where } da \mbox{  is the Haar measure on }G)
% &=&  (f^{\circ} \star g^{\circ} )(x, \xi ).
% \end{eqnarray*}
\end{proof}

We therefore may think of $\ell ^{1}(\Z^{n}, \theta)$ as a
closed $\ast$-subalgebra of $L^{1}(G)$. In particular,  it is a Banach
$\ast$-algebra. Its enveloping $C^*$-algebra is the non-commutative
torus $C^*(\theta )$.

To obtain a concrete realization of the non-commutative torus
$C^*(\theta )$, we consider the regular representation $\lambda $  of 
$\ell ^1(\bZ
^n, \theta )$ on $\ell ^2(\bZ )$ defined by
$$
\lambda (f) g = f \, \natural _\theta \, g \qquad \text{ for } f\in \ell
^1(\bZ ^n), g\in \ell ^2(\bZ ^n) \, .
$$
The regular representation $\lambda $ is faithful, and the  closure of
$\lambda (\ell ^1)$ with respect to the operator norm  is a
$C^*$-algebra  $\cC$. By a special case of
Satz~6 in~\cite{leptin68},  $C^*(\theta )$ is isometrically isomorphic to
$\cC $. From now on, we will therefore not distinguish between  the
abstract algebra $C^*(\theta )$ and its  concrete realization  $\cC $.

\section{Inverse-Closed Subalgebras of $C^*(\theta )$}

Next we construct  a family of inverse-closed Banach subalgebras of
the non-commutative torus $C^*(\theta )$. This construction relies on
two important results in Banach algebra theory and  abstract harmonic
analysis.

  First recall that a Banach $*$-algebra $\cA $  is symmetric, if the spectrum
of every positive element is positive, i.e., $\sigma (a^*a)\subseteq
[0,\infty )$ for all $a\in \cA $. The connection between symmetry and
inverse-closedness is folklore and implicit in many proofs of
symmetry~\cite{GL04a,hulanicki,Lep80,Lud79}. The following proposition is
contained in Palmer's book~\cite[Thm.~11.4.1]{palmer2}. (Since the
regular representation of $\eltw $ is faithful, $\eltw $ is
semisimple, and we may quote a
formulation that is  already adapted to  semisimple  Banach
algebras.)

\begin{prop} \label{p2}
A unital semisimple Banach $\ast$-algebra $\cA $  is symmetric,  if 
and only if it is
inverse-closed in its enveloping $C^{\ast}$-algebra.
\end{prop}

Our  second ingredient is   a fundamental result of
Ludwig~\cite{Lud79}.

\begin{prop} \label{p3}
   If $G$ is a nilpotent group, then $L^1(G)$ is symmetric.
\end{prop}

By combining the explicit construction of non-commutative tori  with
these results, we obtain a
fundamental inverse-closed subalgebra of the non-commutative torus
$C^*(\theta )$.

\begin{tm} \label{tm1}
   The Banach $*$-algebra $\ell ^1(\bZ ^n, \theta )$ is inverse-closed
   in the non-\-commuta\-tive torus $C^*(\theta )$.
\end{tm}

\begin{proof}
   By construction, the central extension $G$ of $\bZ ^n$ is
   nilpotent, and consequently $L^1(G)$ is symmetric by Ludwig's
   result. Lemma~\ref{l1} identifies $\ell ^1(\bZ ^n, \theta )$ with a
   closed $*$-subalgebra of $L^1(G)$, and thus $\ell ^1(\bZ ^n, \theta )$
   is also a symmetric Banach $*$-algebra. By Proposition~\ref{p2} this
   means that $\ell ^1(\bZ ^n, \theta )$ is \ic\ in $C^*(\theta )$, as
   claimed.
\end{proof}

\rem\ For even dimension and a special representation of the
generators of $C^*(\theta )$ by phase-space shifts, Theorem~\ref{tm1} 
was  proved
in~\cite{GL04} for the solution of a problem in \tfa . An earlier
result is contained in~\cite{arveson91}.  See also \cite{Lue09} and
\cite[Ch.~13]{book} for the connections to \tfa .

\vspace{ 2 mm}

To generate more examples of \ic\ subalgebras of $C^*(\theta )$, we
introduce weighted $\ell ^1$-algebras.

Let $v$ be a submultiplicative and symmetric weight function on $\bZ
^n$, i.e., $v$ satisfies the conditions
$$
v(x+y)\leq v(x) v(y) \quad \text{ and }\quad v(-x) = v(x) \quad \text{ for
   all } x,y \in \bZ ^n \, ,
$$
and let $\ell ^1_v (\bZ ^n)$ be the corresponding weighted $\ell
^1$-space with norm $\| f\|_{\ell ^1_v} = \|f v \|_1$. The pointwise
inequality $|(f\, \natural _\theta \, g)(x)| \leq (|f| \ast |g|)(x)$
for all $x\in \zn $ shows that $(\ell ^1_v(\zn ), \, \natural _\theta )$  is a Banach algebra, which we denote $\ell ^1_v (\zn ,
\theta )$.  Since $v$ is symmetric, $\ell ^1_v (\zn , \theta )$ is a
$*$-subalgebra of $\ell ^1 (\zn ,\theta )$.

The next proposition characterizes  those submultiplicative
symmetric
weights for which $\ell ^1_v(\zn , \theta )$ is \ic\ in the
non-commutative torus $C^*(\theta )$.

\begin{prop} \label{p:5}
   The Banach algebra $\ell ^1_v(\zn , \theta )$ is \ic\ in $C^*(\theta
   )$, if and only if $v$ satisfies the Gelfand-Raikov-Shilov condition
   (GRS-condition)
$$
\lim _{n\to \infty } v(nx)^{1/n} = 1 \qquad \text{ for all } x\in \zn
\, .
$$
\end{prop}

\begin{proof}
   Assume first  that $v$ satisfies the GRS-condition.
Then  we may extend $v$ to a weight on  $G$ by setting $\omega (x, 
\xi ) = v(x)$ for all $x \in
\Z^{n}, \xi \in \T$. The extended weight $\omega $  satisfies the GRS-condition
on $G$, so the weighted version of Ludwig's theorem, as proved in
\cite{FGL06}, Theorems 1.3 and 3.4,   implies that $L^1_\omega (G)$ is
symmetric.  Since obviously $\|f^\circ \|_{L^1_\omega (G)} =
\|f\|_{\ell ^1_v}$, Lemma~\ref{l1} implies that $\ell ^1_v( \zn ,
\theta )$ can be identified with a closed subalgebra of $L^1_\omega
(G)$ and thus is  also symmetric. Consequently, by Proposition~\ref{p2}, 
$\ell _{v}^{1} (\Z^{n}, \theta )$ is \ic\ in its enveloping
$C^*$-algebra. To see that this $C^{\ast}$-algebra is $C^{\ast}(\theta)$, it suffices to note that $\ell_{v}^{1}(\Z^{n}, \theta)$ is dense in $\ell^{1}(\Z^{n}, \theta)$ and every $\ast$-representation $\pi$ of $\ell_{v}^{1}(\Z^{n}, \theta)$ on a Hilbert space can be extended to $\ell^{1}(\Z^{n}, \theta)$. The latter follows from the fact that $\pi$ is completely determined by the $\pi(\delta_{x}), x \in \Z^{n}$, and those operators are unitary, so $\overline{\pi}(f) = \sum_{x \in \Z^{n}} f(x) \pi (\delta_{x})$, $f\in \ell^{1}(\Z^{n})$, is the desired extension of $\pi$ to a $\ast$-representation of $\ell^{1}(\Z^{n}, \theta )$.

Conversely,  assume that $v$ violates the GRS-condition. This means that there
exists an $x\in \zn $, such that  $\lim _{n\to \infty } v(nx)^{1/n}
>1$. Since by \eqref{eq:cx} the $n$-th power of $\delta _x$ is of the
form $c_n \delta_{nx}$ with $|c_n|=1$, the spectral radius of $\delta
_x$ in $\ell_{v}^{1} (\Z^{n}, \theta )$ is
$$
r_{\ell_{v}^{1} (\Z^{n}, \theta )}(\delta _x) = \lim _{n\to \infty } 
\|c_n \delta
_{nx}\|_{\ell ^1_v(\zn , \theta )} ^{1/n} = \lim _{n\to \infty }
v(nx)^{1/n} >1 \, .
$$
On the other hand,  since $\delta _x$ is unitary in $\ell_{v}^{1} (\Z^{n}, \theta )$, it is also unitary  in  $C^*(\theta
)$. Consequently, the spectral radius of $\delta _x$ in $C^*(\theta )$
is $1$. Therefore the spectrum of $\delta _x$ in $\ell ^1_v (\zn ,
\theta)$ cannot be equal to the spectrum of $\delta _x $ in
$C^*(\theta )$, and so $\ell ^1_v (\zn ,
\theta)$ is not \ic\ in $C^*(\theta )$.
\end{proof}

\rem\ A  nonspectral subalgebra of the irrational rotation algebra
(the non-commutative torus with two generators) and its simplicity
were  first discussed by  Schweitzer~\cite{Sch94}.

Proposition~\ref{p:5} provides an abundance of examples of \ic\ Banach
subalgebras of a non-commutative torus in higher dimensions. By
taking intersections of weighted $\ell ^1$-algebras, one may now
construct \ic\ Fr\'echet subalgebras of $C^*(\theta )$. In particular,
fix $v(x) = 1+|x|$ for some norm $|\cdot |$ on $\zn $  and set
\begin{equation}
   \label{eq:c4}
   \mathcal{S}(\zn , \theta ) = \bigcap _{s\geq 0} \ell ^1_{v^s} (\zn , \theta ) = \{
   f \in \ell ^1(\zn ): |f(x)| = \cO ( |x|^{-s}) \, \forall s\geq 0\}
   \, .
\end{equation}
Then $\mathcal{S}(\zn , \theta )$ consists of all rapidly decreasing sequences
and coincides with  the usual smooth non-commutative torus.  Since an arbitrary
intersection of \ic\ subalgebras is again \ic , $  \mathcal{S}(\zn , \theta )$
is an \ic\ Fr\'echet algebra  of the non-commutative torus  $C^*(\theta
)$. This result goes back to Connes~\cite{connes80}.

Proposition~\ref{p:5} yields  \ic\ subalgebras of $C^*(\theta )$ that
are even smaller than $  \mathcal{S}(\zn , \theta )$. For this fix a
subexponential weight $v(x) = e^{a |x|^b}$ with  $a>0$ and $0<b<1$. Then
$v$ satisfies the GRS-condition, and thus $\ell ^1_v(\zn , \theta )$
is \ic\ in $C^*(\theta )$. On the other hand, $\ell ^1_v(\zn , \theta
)$ is a Banach subalgebra of the smooth non-commutative torus $\mathcal{S}(\zn ,
\theta )$. In the language of non-commutative geometry, one might say
that $\ell ^1_v(\zn , \theta )$ consists of ``ultra-smooth'' elements
of $C^*(\theta )$.

\section{Simplicity } %of Non-Commutative Tori}

  The construction of \ic\ subalgebras of non-commutative tori is
completely independent of the fine structure of these tori. In
particular, the simplicity of $\eltw  $ is not related to its
spectral properties.

In this section we treat the question when the  twisted
$\ell ^1$-algebra  $\ell ^1  (\zn , \theta )$ is  simple. Making use
of  the symmetry of $\ell ^1(\zn , \theta )$, one can derive 
 Theorem~\ref{simp} below from the characterization of the
simplicity of higher-dimensional non-commutative tori $C^*(\theta
)$ in~\cite{phil06}, but one has to  go back to \cite{sla72}
and \cite{ell80} for its proof. We offer a simplified  proof that
works directly for $\ell ^1 (\zn , \theta )$, from which  the known
result about $C^*(\theta )$ follows. Our proof for the twisted $\ell
^1$-algebras is fairly elementary, but its idea is probably old.

Let $\delta _m, m\in \mathbb{Z}^n,$ denote the ``Dirac'' functions on
$\bZ ^n$  and
$e_j, j=1, \dots , n,$ the standard basis of $\mathbb{Z}^n$. Then
$\delta _m$ is central in $\ell ^1 (\mathbb{Z}^n, \natural, \ast )$,
\fif\ $\delta _m \, \natural \, \delta _{e_j} = \delta _{e_j} \,
\natural \,  \delta _m  $ for $j=1, \dots ,n$. Since
$$
  \delta _m \, \natural \, \delta _{e_j}  = \sigma (m, e_j) \delta
  _{m+e_j} 
=  \theta _{n,j} ^{m_n} \dots \theta _{j+1, j} ^{m_{j+1} }\, \delta
_{m+e_j}
$$
and 
$$
\delta _{e_j} \, \natural \, \delta _m = \sigma  (e_j, m) \delta
_{m+e_j} = \theta _{j,1}^{m_1} \dots \theta _{j,j-1} ^{m_{j-1}} \, \delta
  _{m+e_j} \, ,
$$
the following conditions are equivalent: 
\begin{enumerate}
\item[(i)] $\delta _m$ is central in $\eltw $.
\item[(ii)] $\sigma (m,e_j) = \sigma (e_j, m)$ for $j=1,\dots ,n$.
\item [(iii)] $\prod _{j=1}^n \theta _{jk}^{m_j} = 1$ for $k=1,\dots ,n$.
\end{enumerate}
If $\vartheta = (\vartheta _{jk})$ is a (non-unique) skew-symmetric
real matrix with $e^{2\pi i \vartheta _{jk}}  = \theta _{jk}$, then
(iii) means that $\sum_{j=1}^n m_j \vartheta _{jk} \in \mathbb{Z}$
for $k=1,\dots , n$. So, denoting the skew-symmetric bilinear form
$(m,l)\to \vartheta (l,m) = m^T \vartheta l$ by $\vartheta $ again, a  fourth
equivalent property is

\begin{enumerate}
\item[(iv)]  $\vartheta (l,m) \in \mathbb{Z} \quad
\text{ for all } l\in \mathbb{Z}^n \, .$
\end{enumerate}

\begin{definition}
  A cocycle $\sigma $ is called degenerate, if there exists a
  \textrm{non-zero} $m\in   \mathbb{Z}^n$ satisfying one of the
  equivalent  conditions (i) --- (iv). Otherwise $\sigma $ is called
  nondegenerate. 
\end{definition}
We note that $\sigma $ can be nondegenerate even if $\vartheta$ is
degenerate in the sense of linear algebra. 

\rem\ It is well known that  a unital Banach algebra $\cA $ with
non-trivial center  is not simple. For if the center $\cZ $ is non-trivial,
i.e., its dimension is at least two, then it contains an
element $a $ that is not invertible in $\cZ $  by the theorem of
Gelfand-Mazur. Since $\cZ $ is inverse-closed in $\cA $, $a$ is not
invertible in $\cA $. Consequently the
generated ideal $a\cA = \cA a$ is a proper two-sided ideal, and so is its
closure $\overline{a\cA}$. Thus $\cA $ is not simple.

The following theorem characterizes the simplicity of  twisted
$\ell ^1$-algebras.

\begin{tm} \label{simp}
Let $v$ be an arbitrary, submultiplicative weight function on $\bZ
^n$($v$ need not satisfy the GRS-condition).
 
Then the  algebra $\ell ^1_v(\zn , \theta )$ is simple, \fif\ the
 cocycle $\sigma  $ is nondegenerate.
\end{tm}

\begin{proof}
   % For $j=1, \dots , n$ we denote by $\delta _j$ the ``Dirac'' function
   % of the point $(0, \dots, 0, 1, 0 , \dots )$ where $1$ is in the $j$-th
   % coordinate.

If $\sigma $ is degenerate, then $\ell ^1_v(\zn , \theta )$ has a
non-trivial center and is not simple by the  remark above. 

Now suppose that $\sigma  $ is nondegenerate. For each $j \in \{ 1,
\dots , n\}$ the element  $\delta _{e_j}$ is
unitary, and its adjoint is $\sigma (-e_j, e_j) \delta _{-e_{j}}$. For arbitrary $x\in \zn $
and $k\in \bN $ we have
$$
(\delta _{e_j} ^*)^k \, \natural \, \delta _x \, \natural \, \delta _{e_j} ^k
= \beta _x ^k \delta _x
$$
for some $\beta _x \in \bT $. More precisely, $\beta _x = 1$, \fif\
$\delta _x$ commutes with $\delta _{e_j}$. We denote the ''centralizer'' of
$\delta _{e_j}$ by
$$
C_j = \{ y\in \zn : \delta _y \, \natural \, \delta _{e_j} = \delta _{e_j} \,
\natural \, \delta _y \} \, .
$$

  Now let $I$ be a (closed) two-sided
ideal of $\ell ^1_v(\zn , \theta  )$ and $f= \sum _{x\in \zn } \alpha _x \delta _x \in I
\subseteq \ell ^1_v(\zn , \theta ) $. We consider the behavior of the averages
$$
J_m (f) = \frac{1}{m} \sum _{k=1}^m (\delta _{e_j}^*)^k \twist f \twist
\delta _{e_j}^k = \sum _{x\in \zn }  \alpha _x \Big( \frac{1}{m} \sum
_{k=1}^m \beta _x ^k \Big) \delta _x \, .
$$
If $x \in C_j$, then  $\frac{1}{m} \sum
_{k=1}^m \beta _x ^k = 1$; if $x \not \in C_j$, then the average 
$\frac{1}{m} \sum
_{k=1}^m \beta _x ^k  $ converges to zero for $m\to \infty $.
Using dominated convergence, we conclude that
$$
\lim _{m\to \infty } J_m (f) = \sum _{x\in C_j} \alpha _x \delta _x = 
f \, \chi _{C_j}
$$
with convergence in the  $\ell ^1_v$-norm.  For $f\in I $, this means
that also $f \chi _{C_j} \in I$. Since this is true for all $j=1,
\dots ,n$, we obtain that $f (\prod _{j=1}^n \chi _{C_j}) = f \chi
_{\bigcap _{j=1}^n C_j } \in I$.

Since $\sigma  $ is nondegenerate, we must have $  \bigcap _{j=1}^n C_j  =
\{ 0 \}$ and thus $f(0) \delta _0 \in I$. Either $I = \ell ^1_v(\zn , \theta ) $ or $I $
is a proper ideal and $f(0)=0$. By applying the argument to
  $\delta _x \twist f\in I$ for every  $x\in \zn $, we obtain that $ (\delta
  _x \twist f)(0) =  \sigma  (x,-x) \, f(-x) = 0$, so $f(x) = 0$  for
  all $x\in \zn $. Consequently  either $I= \ell ^1_v(\zn , \theta  )$
  or $I = \{0\}$,  and thus 
$  \ell ^1_v(\zn , \theta  )$ is simple.
\end{proof}

\rem\ We may also  obtain an alternative proof of the well-known
$C^*$-analogue of Theorem~\ref{simp}.
The above proof also  works  for $C^*(\theta )$,  because
the finitely supported functions are dense in $C^*(\theta )$ and the
inner automorphisms are also isometric in the $C^*(\theta )$-norm.

% (ii) If $\sigma  $ is degenerate, then  the ``only if''-part of
% Theorem~\ref{simp} asserts that $\eltw $ is not simple and thus
% contains a non-trivial closed two-sided ideal $I$. Since $\eltw $ is
% symmetric, there exists a non-zero positive linear functional $\phi $
% on $\eltw $ vanishing on $I$. Now let $\pi _\phi $ be the
% $*$-representation obtained from $\phi $ by the
% GNS-construction. Then we have $\pi _\phi (I) = \{ 0\}$.  So the
% kernel of $\pi _\phi $ in $C^*(\theta )$ contains
% $I$ and hence  is a non-trivial,  closed, two-sided ideal in  $C^*(\theta )$. Therefore  $C^*(\theta )$ is not simple.

\def\cprime{$'$} \def\cprime{$'$} \def\cprime{$'$} \def\cprime{$'$}
   \def\cprime{$'$}

%Further citations: \cite{phil06,Lep68,Lep67,Lue06,rieffel88,ell80,sla72}

  \bibliographystyle{abbrv}
  \bibliography{general,new}

\begin{thebibliography}{10}

\bibitem{arveson91}
W.~Arveson.
\newblock Discretized {CCR} algebras.
\newblock {\em J. Operator Theory}, 26(2):225--239, 1991.

\bibitem{Bad98}
C.~Badea.
\newblock The stable rank of topological algebras and a problem of {R}. {G}.
   {S}wan.
\newblock {\em J. Funct. Anal.}, 160(1):42--78, 1998.

\bibitem{connes80}
A.~Connes.
\newblock {$C\sp{\ast} $} alg\`ebres et g\'eom\'etrie diff\'erentielle.
\newblock {\em C. R. Acad. Sci. Paris S\'er. A-B}, 290(13):A599--A604, 1980.

\bibitem{ell80}
G.~A. Elliott.
\newblock On the {$K$}-theory of the {$C^{\ast} $}-algebra generated by a
   projective representation of a torsion-free discrete abelian group.
\newblock In {\em Operator algebras and group representations, {V}ol. {I}
   ({N}eptun, 1980)}, volume~17 of {\em Monogr. Stud. Math.}, pages 157--184.
   Pitman, Boston, MA, 1984.

\bibitem{FGL06}
G.~Fendler, K.~Gr{\"o}chenig, and M.~Leinert.
\newblock Symmetry of weighted {$L\sp 1$}-algebras and the {GRS}-condition.
\newblock {\em Bull. London Math. Soc.}, 38(4):625--635, 2006.

\bibitem{book}
K.~Gr{\"o}chenig.
\newblock {\em Foundations of time-frequency analysis}.
\newblock Birkh\"auser Boston Inc., Boston, MA, 2001.

\bibitem{Gr10}
K.~Gr\"ochenig.
\newblock Wiener's lemma: Theme and variations. An introduction to spectral
   invariance.
\newblock In B.~Forster and P.~Massopust, editors, {\em Four Short Courses on
   Harmonic Analysis}, Appl. Num. Harm. Anal. Birkh\"auser, Boston, 2010.

\bibitem{GL04}
K.~Gr{\"o}chenig and M.~Leinert.
\newblock Wiener's lemma for twisted convolution and {G}abor frames.
\newblock {\em J. Amer. Math. Soc.}, 17:1--18, 2004.

\bibitem{GL04a}
K.~Gr{\"o}chenig and M.~Leinert.
\newblock Symmetry and inverse-closedness of matrix algebras and functional
   calculus for infinite matrices.
\newblock {\em Trans. Amer. Math. Soc.}, 358(6):2695--2711 (electronic), 2006.

\bibitem{hulanicki}
A.~Hulanicki.
\newblock On the spectrum of convolution operators on groups with polynomial
   growth.
\newblock {\em Invent. Math.}, 17:135--142, 1972.

% \bibitem{Lep67}
% H.~Leptin.
% \newblock Verallgemeinerte {$L^{1}$}-{A}lgebren und projektive 
%{D}arstellungen
%   lokal kompakter {G}ruppen. {I}, {II}.
% \newblock {\em Invent. Math. 3 (1967), 257-281; ibid.}, 4:68--86, 1967.

\bibitem{leptin68}
H.~Leptin.
\newblock Darstellungen verallgemeinerter {$L^{1}$}-{A}lgebren.
\newblock {\em Invent. Math.}, 5:192--215, 1968.

\bibitem{Lep80}
H.~Leptin.
\newblock The structure of {$L\sp{1}(G)$} for locally compact groups.
\newblock In {\em Operator algebras and group representations, Vol. II (Neptun,
   1980)}, volume~18 of {\em Monogr. Stud. Math.}, pages 48--61. Pitman, Boston,
   MA, 1984.

\bibitem{Lud79}
J.~Ludwig.
\newblock A class of symmetric and a class of {W}iener group algebras.
\newblock {\em J. Funct. Anal.}, 31(2):187--194, 1979.

\bibitem{Lue09}
F.~Luef.
\newblock Projective modules over noncommutative tori are multi-window {G}abor
   frames for modulation spaces.
\newblock {\em J. Funct. Anal.}, 257(6):1921--1946, 2009.

\bibitem{mackey58}
G.~W. Mackey.
\newblock Unitary representations of group extensions. {I}.
\newblock {\em Acta Math.}, 99:265--311, 1958.

\bibitem{palmer2}
T.~W. Palmer.
\newblock {\em Banach algebras and the general theory of {$*$}-algebras. {V}ol.
   2}, volume~79 of {\em Encyclopedia of Mathematics and its Applications}.
\newblock Cambridge University Press, Cambridge, 2001.
\newblock $*$-algebras.

\bibitem{phil06}
N.~C. Phillips.
\newblock Every simple higher dimensional noncommutative torus is an AT
   algebra.
\newblock  2006,
\newblock arXiv:math/0609783v1.

\bibitem{rieffel88}
M.~A. Rieffel.
\newblock Projective modules over higher-dimensional noncommutative tori.
\newblock {\em Canad. J. Math.}, 40(2):257--338, 1988.

\bibitem{sch26}
O.~Schreier.
\newblock \"{U}ber die {E}rweiterung von {G}ruppen {I}.
\newblock {\em Monatsh. Math. Phys.}, 34(1):165--180, 1926.

\bibitem{Sch94}
L.~B. Schweitzer.
\newblock A nonspectral dense {B}anach subalgebra of the irrational rotation
   algebra.
\newblock {\em Proc. Amer. Math. Soc.}, 120(3):811--813, 1994.

\bibitem{sla72}
J.~Slawny.
\newblock On factor representations and the {$C^{\ast} $}-algebra of canonical
   commutation relations.
\newblock {\em Comm. Math. Phys.}, 24:151--170, 1972.

\end{thebibliography}

\end{document}